\documentclass[a4paper]{amsart}
\usepackage{a4wide}
\usepackage{microtype}
\usepackage{amsmath}
\usepackage{amsthm}
\usepackage{amsfonts}
\usepackage{amssymb}
\usepackage{mathrsfs}
\usepackage{tikz}
\usepackage[alphabetic,initials,nobysame,non-compressed-cites]{amsrefs}
\usepackage[hidelinks]{hyperref}
\usepackage{graphicx}
\usepackage{overpic}
\usepackage{colortbl}
\usepackage{todonotes}
\usepackage{enumitem}
\usepackage{tabu}
\usepackage{pdflscape}

\theoremstyle{plain}
\newtheorem*{theorem}{Theorem}
\newtheorem*{lemma}{Lemma}
\newtheorem*{conjecture}{Conjecture}
\newtheorem*{proposition}{Proposition}

\theoremstyle{definition}
\newtheorem*{problem}{Problem}

\theoremstyle{remark}
\newtheorem{remark}{Remark}

\title[Packings with high chromatic number]{Ball packings with high chromatic numbers\\from strongly regular graphs}
\author{Hao Chen}
\address{Departement of Mathematics and Computer Science, Technische Universiteit Eindhoven}
\email{hao.chen@tue.nl}


\thanks{The author was supported by the ERC Advanced Grant number 247029
``SDModels'' and by NWO/DIAMANT grant number 613.009.031.  Part of the research
was conducted while the author was at Freie Universit\"at Berlin.}

\begin{document}

\begin{abstract}
	Inspired by Bondarenko's counter-example to Borsuk's conjecture, we notice
	some strongly regular graphs that provide examples of ball packings whose
	chromatic numbers are significantly higher than the dimensions.  In
	particular, from generalized quadrangles we obtain unit ball packings in
	dimension $q^3-q^2+q$ with chromatic number $q^3+1$, where $q$ is a prime
	power.  This improves the previous lower bound for the chromatic number of
	ball packings.
\end{abstract}

\maketitle

\section{The problem and previous works}\label{sec:problem}
A \emph{ball packing} in $d$-dimensional Euclidean space is a collection of
balls with disjoint interiors.  The \emph{tangency graph} of a ball packing
takes the balls as vertices and the tangent pairs as edges.  The
\emph{chromatic number} of a ball packing is defined as the chromatic number of
its tangency graph.

The Koebe--Andreev--Thurston disk packing theorem says that every planar graph
is the tangency graph of a $2$-dimensional ball packing.  The following
question is asked by Bagchi and Datta in~\cite{bagchi2013} as a higher
dimensional analogue of the four-colour theorem:
\begin{problem}
	What is the maximum chromatic number $\chi(d)$ over all the ball packings in
	dimension~$d$?
\end{problem}
The authors gave $d+2\le\chi(d)$ as a lower bound since it is easy to construct
$d+2$ mutually tangent balls.  By ordering the balls by size, the authors also
argued that $\kappa(d)+1$ is an upper bound, where $\kappa(d)$ is the kissing
number for dimension $d$.

However, the case of $d=3$ has already been investigated by
Maehara~\cite{maehara2007}, who proved that $6\le\chi(3)\le13$.  His
construction for the lower bound uses a variation of Moser's spindle, which is
the tangency graph of an \emph{unit} disk packing in dimension $2$ with
chromatic number $4$, and the following lemma:
\begin{lemma}
	If there is a unit ball packing in dimension $d$ with chromatic number
	$\chi$, then there is a ball packing in dimension $d+1$ with chromatic number
	$\chi+2$.
\end{lemma}
The technique of Maehara~\cite{maehara2007} can be easily generalized to higher
dimensions and gives $d+3\le\chi(d)$.

Another progress is made by Cantwell in an answer on
MathOverflow~\cite{mathoverflow}, who proved that the graph of the halved
$5$-cube (also called the Clebsch graph) is the tangency graph of a
$5$-dimensional \emph{unit} ball packing with chromatic number $8$.  Then the
Lemma implies that $10\le\chi(6)$.  This argument can be generalized to higher
dimensions using a result of Linial, Meshulam and
Tarsi~\cite{linial1988}*{Theorem~4.1}, and gave $d+4\le\chi(d)$ for $d=2^k-2$.

As we have seen, both constructions study the chromatic number of unit ball
packings and invoke the Lemma.  We will do the same.  A unit ball packing can
be regarded as a set of points such that the minimum distance between pairs of
points is at least $1$, then the tangency graph of the packing is the
unit-distance graph for these points.  The \emph{finite version} of the Borsuk
conjecture can be formulated as follows: the chromatic number of the
unit-distance graph for a set of points with maximum distance $1$ is at most
$d+1$.  So the chromatic number problem for unit ball packings is the
``opposite'' of the Borsuk conjecture.  By ordering the unit balls by height,
we see that the chromatic number of a unit ball packing is at most one plus the
one-side kissing number.

The Borsuk conjecture was first disproved by Kahn and Kalai~\cite{kahn1993}.
Recently, Bondarenko~\cite{bondarenko2014} found a counter-example for Borsuk
conjecture in dimension $65$. His construction was then slightly improved by
Jenrich~\cite{jenrich2014} to dimension $64$, which is the current record for
the smallest counter-example.  Their construction is based on geometric
representations of strongly regular graphs.

In this note, we use the technique of Bondarenko to find unit ball packings
with strongly regular tangency graphs, whose chromatic numbers are
significantly higher than their dimensions.  In particular
\begin{theorem}
	For every prime power $q$, there is a unit ball packing of dimension
	$d=q^3-q^2+q$ whose tangency graph is strongly regular with chromatic number
	$\chi(d)=q^3+1$.
\end{theorem}
Examples are given by the graphs of generalized quadrangles with parameters
$(q,q^2)$.  This yields the first non-constant lower bound for the difference
$\chi(d) - d$.

In view of~\cite{kahn1993}, we propose the following conjecture, and hope that
examples in this note may help further improvement of the lower bound.
\begin{conjecture}
	There is a constant $c$ such that $\chi(d)\ge c^{\sqrt d}$.
\end{conjecture}

\subsection*{Acknowledgement} 

I would like to thank Yaokun Wu and Eiichi Bannai for the dinner at Shanghai
Jiao Tong University, where the idea of this note originated.  I'm also
grateful to Martin Aigner, Michael Joswig and Bhaskar Bagchi for helpful
suggestions and comments, and David Roberson and anonymous referees for
pointing out mistakes in preliminary versions of the note.

\section{Strongly regular graphs}\label{sec:prepare}
We use \cite{cameron1991} for general references on strongly regular graphs.

Let $G$ be a strongly regular graph with parameters $(v, k, \lambda, \mu)$.
That is, $G$ is a $k$-regular graph on~$v$ vertices such that every pair of
adjacent vertices have $\lambda$ neighbors in common and every pair of
non-adjacent vertices have $\mu$ neighbors in common.  We assume that
\begin{equation}\label{eq:assume}
	\lambda-\mu\ge-2k/(v-1).
\end{equation}
If this is not the case, we may replace $G$ by its complement $\bar G$, which
is a strongly regular graph with parameters $(v, v-k-1, v-2k-2+\mu,
v-2k+\lambda)$.  For our study of ball packings, we may focus on connected
graphs, therefore $\mu>0$.  For any vertex of $G$, the graphs induced by its
neighbors in $G$ and by its neighbors in $\bar G$ are respectively the
\emph{first} and the \emph{second subconstituent} of~$G$.

The adjacency matrix $A$ of $G$ has three eigenvalues $k$, $r$, $s$ with
multiplicities $1$, $f$, $g$, respectively.  They can be expressed in terms of
the parameters as follows:
\begin{align*}
	r,s&=(\lambda-\mu\pm\delta)/2,\\
	f,g&=(v-1\pm\Delta)/2,
\end{align*}
where $\delta=\sqrt{(\lambda-\mu)^2+4(k-\mu)}$ and
$\Delta=((v-1)(\mu-\lambda)-2k)/\delta\le 0$.  The eigenvalues of $\bar G$ are
$v-k-1$, $-s-1$, $-r-1$ with multiplicities $1$, $g$, $f$, respectively.  Note
that $r>0>s+1$ and $f\le g$.  

Let $I$ be the identity matrix and $J$ the all-ones matrix.  Then
\[
	E=(A-sI)(I-J/v)
\]
is an eigenmatrix of $A$ corresponding to the eigenvector $r$, and the column
vectors of $E$ (labeled by vertices of $G$) form a unit spherical $2$-distance
set on the sphere $\mathbb{S}^{f-1}\subset\mathbb{R}^f$, with angles
$\cos\alpha=r/k$ for adjacent vertices and $\cos\beta=-(r+1)/(v-k-1)$ for
non-adjacent vertices~\cite{bondarenko2014}; see
also~\cite[Theorem~4.1.4]{brouwer1989}.  By putting a ball of radius
$\sin(\alpha/2)=\sqrt{(1-r/k)/2}$ at each point of the $2$-distance set, we
obtain a ball packing whose tangency graph is $G$.

By Hoffman's bound~\cite{hoffman1970} (see also \cite{delsarte1973}
\cite[Proposition~1.3.2]{brouwer1989}), the clique number of the complement
$\omega(\bar G)$ is at most $1+(v-k-1)/(1+r)$, so the chromatic number
$\chi(G)$ is at least 
\[
	v/\big(1+\frac{v-k-1}{1+r}\big)=1-k/s.
\]

\begin{remark}
	As B.\ Bagchi pointed to the author, the same argument works for any regular
	graph.  More specifically, if the adjacency matrix of a $k$-regular graph $G$
	of $v$ vertices has the minimum eigenvalue $s$ of multiplicity $g$, then $G$
	is the tangency graph of a ball packing in dimension $v-g-1$, and the
	chromatic number of $G$ is at least $1-k/s$.  The advantage of strongly
	regular graphs is that $g$ tends to be big, which would decrease the
	dimension.
\end{remark}

\section{Colorful strongly regular ball packings}\label{sec:result}
From Brouwer's online list of strongly regular graphs~\cite{brouwer}, we notice
some graphs such that $f+3 < 1-k/s$.  In Table~\ref{tab}, we list their
parameters, eigenvalues with multiplicities, and the Hoffman bound $1-k/s$.
For comparison, we highlight the dimension $f$ and the Hoffman bound $1-k/s$.
For the complement of McLaughlin graph, the Hoffman bound gives the exact value
of the chromatic number~\cite{haemers1995}.

\begin{remark}
	For many graphs in the table, the conventional parameters in the literature
	do not satisfy our assumption~\eqref{eq:assume}.  For these graphs, we use
	the parameters of their complements, as explained in the beginning of
	Section~\ref{sec:prepare}.
\end{remark}

\begin{table}[hbt]
	\scriptsize
	\[
		\begin{tabu}{| r | c c c c | r >{\columncolor[gray]{0.8}}l | r l | >{\columncolor[gray]{0.8}}c |}
		  \text{name or ref} &
			v & k & \lambda & \mu &
			r & f & s & g &
			1-k/s \\

			\hline

      %
      %
			\text{Higman--Sims} &
			100 & 77 & 60 & 56 &
			7 & 22 & -3 & 77 &
		 	80/3 \\

      %
			\text{\cite{degraer2008,leemans2013}} &
			105 & 72 & 51 & 45 &
			9 & 20 & -3 & 84 &
			25 \\

			\text{\cite{degraer2008,leemans2013}} &
      120 & 77 & 52 & 44 &
      11 & 20 & -3 & 99 &
      80/3 \\

      \text{\cite{haemers2002}} &
      126 & 75 & 48 & 39 &
      12 & 20 & -3 & 105 &
      26 \\

			\text{2nd subconst. of McL} &
			162 & 105 & 72 & 60 &
			15 & 21 & -3 & 140 &
			36 \\

      \text{\cite{haemers1981}} &
      175 & 102 & 65 & 51 &
      17 & 21 & -3 & 153 &
      35 \\

			\text{\cite{haemers1981, degraer2008}} &
      176 & 105 & 68 & 54 &
      17 & 21 & -3 & 154 &
      36 \\

      \text{\cite{haemers1981}} &
      176 & 85 & 48 & 34 &
      17 & 22 & -3 & 153 &
      88/3 \\

      \text{\cite{delsarte1973}} &
      243 & 132 & 81 & 60 &
      24 & 22 & -3 & 220 &
      45 \\

			\text{\cite{degraer2008,leemans2013}} &
      253 & 140 & 87 & 65 &
      25 & 22 & -3 & 230 &
      143/3 \\

			\text{McLaughlin} &
			275 & 162 & 105 & 81 &
			27 & 22 & -3 & 252 &
			55 \\

      \text{\cite{goethals1975}} &
      276 & 135 & 78 & 54 &
      27 & 23 & -3 & 252 &
      46 \\

      \text{\cite{bondarenko2013}} &
      729 & 520 & 379 & 350 &
      22 & 112 & -5 & 616 &
      621/5 \\

    \end{tabu}
  \]
  \caption{
    Strongly regular graphs with high chromatic numbers\label{tab}
  }
\end{table}

Two infinite families are not included in the table. One is the complements to
the C20 family from Hubaut~\cite{hubaut1975}, recovered by
Godsil~\cite{godsil1992}*{Lemma~5.3}, with the parameters
\[
	(q^3, (q+1)(q^2-1)/2, (q+3)(q^2-3)/4+1, (q+1)(q^2-1)/4)
\]
where $q$ is an odd prime power.  They can be represented as \emph{unit} ball
packings in dimension $f=q^2-q$ with chromatic number at least $1-k/s=q^2$,
which already provide a non-constant difference $\chi(d)-d$.

The other family is the complements to the point graphs of the generalized
quadrangles with parameters $(q,q^2)$, which provides an even better
difference.

A \emph{generalized quadrangle}~\cite{payne2009} with parameters
$(\sigma,\tau)$, denoted by $GQ(\sigma,\tau)$, is an incidence structure
$(P,L,\in)$, where $P$ is the set of points and $L$ is the set of lines,
satisfying the following axioms:
\begin{itemize}
	\item Each point is incident with $\tau+1$ lines and two distinct points
		are incident with at most one line.
	\item Each line is incident with $\sigma+1$ points and two distinct lines
		are incident with at most one point.
	\item For a point $p$ and a line $\ell$ such that $p\notin\ell$, there is a
		unique pair $(p',\ell')$ such that $p\in\ell'\ni p'\in\ell$.
\end{itemize}
It is known that $GQ(q,q^2)$ exists when $q$ is a prime power, and is unique
for $q=2,3$.  

The \emph{point graph} of a generalized quadrangle, also denoted by
$GQ(\sigma,\tau)$, has $P$ as the set of vertices, and two vertices are joined
by an edge if they are incident to the same line.  It is a strongly regular
graph with parameters
\[
	((\sigma\tau+1)(\sigma+1), \sigma(\tau+1), \sigma-1, \tau+1).
\]
The complement of $GQ(q,q^2)$, denoted by $\overline{GQ}(q,q^2)$, is a strongly
regular graph with parameters
\[
	((q+1)(q^3+1), q^4, q(q-1)(q^2+1), (q-1)q^3).
\]
In particular, $\overline{GQ}(2,4)$ is the Schl\"afli graph and $GQ(3,9)$ is
the first subconstituent of the McLaughlin graph ($112$ vertices).
$\overline{GQ}(q,q^2)$ can be represented as a ball packing in dimension
$f=q^3-q^2+q$, and the chromatic number of $\overline{GQ}(q,q^2)$ is at
least~$q^3+1$.  Moreover, 

\begin{proposition}
	The chromatic number of $\overline{GQ}(q,q^2)$ is exactly $q^3+1$.
\end{proposition}

\begin{proof}
	A \emph{spread} of a generalized quadrangle is a set of lines such that each
	point is incident with a unique line in the set.
	By~\cite{payne2009}*{Theorem~3.4.1(ii)}, a generalized quadrangle $GQ(q,q^2)$
	has spreads, meaning that the vertices of the graph $GQ(q,q^2)$ can be
	partitioned into $q^3+1$ cliques of size $q+1$.  So the chromatic number of
	$\overline{GQ}(q,q^2)$ is at most~$q^3+1$.
\end{proof}

This proves the Theorem.  By the Lemma in Section~\ref{sec:problem}, we have
also constructed a ball packing in dimension $q^3-q^2+q+1$ with chromatic
number $q^3+3$.  This is, to the knowledge of the author, the first
construction with non-constant difference $\chi(d)-d$.

\begin{remark}
	Up to complement, the Clebsch graph, the Higman--Sims graph, the McLaughlin
	graph, the second subconstituent of the McLaughlin graph ($162$ vertices),
	and $\overline{GQ}(q,q^2)$ are the only known examples of Smith
	graphs~\cite{smith1975, cameron1978}.  They can be constructed from a rank
	$3$ permutation group such that the stabilizer of a vertex has rank $\le 3$
	on both subconstituents.  It turns out that Smith graphs tend to have high
	chromatic number (the Clebsch graph is noticed by
	Cantwell~\cite{mathoverflow}).
\end{remark}

\begin{remark}
	Hoffman's bound is not always a good bound for the chromatic number.  Our
	method does not recover, for example, Bondarenko's counter-example to the
	Borsuk conjecture.  The author believes that there are other strongly regular
	graphs with high chromatic number.  However, for the conjecture in
	Section~\ref{sec:problem}, the power of strongly regular graphs might be very
	limited.
\end{remark}


\bibliography{References}

\end{document}